\documentclass[11pt]{amsart}
\usepackage{graphicx,amscd,color}

\theoremstyle{plain}
\newtheorem{theorem}{Theorem}[section]

\newtheorem{lemma}[theorem]{Lemma}

\theoremstyle{remark}

\newtheorem{claim}{Claim}

\title [Bridge spheres for the unknot are topologically minimal]
{Bridge spheres for the unknot are topologically minimal}

\author[J. H. Lee]{Jung Hoon Lee}
\address{Department of Mathematics and Institute of Pure and Applied Mathematics,
Chonbuk National University, Jeonju 561-756, Korea}
\email{junghoon@jbnu.ac.kr}

\begin{document}

\begin{abstract}
We show that an $(n+1)$-bridge sphere for the unknot is a topologically minimal surface of index at most $n$.
\end{abstract}

\maketitle

\section{Introduction}\label{sec1}

Let $S$ be a closed orientable separating surface embedded in a $3$-manifold $M$.
The structure of the set of compressing disks for $S$,
such as how a pair of compressing disks in opposite sides of $S$ intersects,
reveals some topological properties of $M$.
For example, if $S$ is a minimal genus Heegaard surface of an irreducible manifold $M$ and
$S$ has a pair of disjoint compressing disks in opposite sides,
then $M$ contains an incompressible surface \cite{Casson-Gordon}.

A {\em disk complex} $\mathcal{D}(S)$ of $S$ is a simplicial complex defined as follows.

\begin{itemize}
\item Vertices of $\mathcal{D}(S)$ are isotopy classes of compressing disks for $S$.
\item A collection of $k+1$ vertices forms a $k$-simplex if there are pairwise disjoint representatives.
\end{itemize}

The disk complex of an incompressible surface is empty.
A surface $S$ is {\em strongly irreducible} if $S$ compresses to both sides and
every compressing disk for $S$ in one side intersects every compressing disk in the opposite side.
So the disk complex of a strongly irreducible surface is disconnected.
Extending these notions, Bachman defined topologically minimal surfaces \cite{Bachman},
which can be regarded as topological analogues of (geometrically) minimal surfaces.

A surface $S$ is {\em topologically minimal} if
$\mathcal{D}(S)$ is empty or $\pi_i(\mathcal{D}(S))$ is non-trivial for some $i$.
The {\em topological index} of $S$ is $0$ if $\mathcal{D}(S)$ is empty, and
the smallest $n$ such that $\pi_{n-1}(\mathcal{D}(S))$ is non-trivial, otherwise.

Topologically minimal surfaces share some useful properties.
For example, if an irreducible manifold contains a topologically minimal surface and an incompressible surface,
then the two surfaces can be isotoped so that any intersection loop is essential in both surfaces.
There exist topologically minimal surfaces of arbitrary high index \cite{Bachman-Johnson},
and see also \cite{Lee} for possibly high index surfaces in (closed orientable surface)$\times I$.
In this paper we consider bridge splittings of $3$-manifolds, and
show that the simplest bridge surfaces, bridge spheres for the unknot in $S^3$, are topologically minimal.
The main idea is to construct a retraction
from the disk complex of a bridge sphere to $S^{n-1}$ as in \cite{Bachman-Johnson} and \cite{Lee}.

\begin{theorem}\label{thm}
An $(n+1)$-bridge sphere for the unknot is a topologically minimal surface of index at most $n$.
\end{theorem}

In particular, the topological index of a $3$-bridge sphere for the unknot is two.
We conjecture that the topological index of an $(n+1)$-bridge sphere for the unknot is $n$.
There is another conjecture that the topological index of a genus $n$ Heegaard surface of $S^3$ is $2n - 1$.
This correspondence is maybe due to the fact that
a genus $n$ Heegaard splitting of $S^3$ can be obtained
as a $2$-fold covering of $S^3$ branched along an unknot in $(n+1)$-bridge position.

\section{Bridge splitting}\label{sec2}

For a closed $3$-manifold $M$, a {\em Heegaard splitting} $M = V^+ \cup_S V^-$
is a decomposition of $M$ into two handlebodies $V^+$ and $V^-$ with $\partial V^+ = \partial V^- = S$.
The surface $S$ is called a {\em Heegaard surface} of the Heegaard splitting.

Let $K$ be a knot in $M$ such that $V^{\pm} \cap K$ is
a collection of $n$ boundary-parallel arcs $\{ a_1^{\pm}, \ldots, a_n^{\pm} \}$ in $V^{\pm}$.
Each $a_i^{\pm}$ is called a {\em bridge}.
The decomposition $(M, K) = (V^+, V^+ \cap K) \cup_S (V^-, V^- \cap K)$ is called
a {\em bridge splitting} of $(M, K)$, and
we say that $K$ is in {\em $n$-bridge position} with respect to $S$.
A bridge $a_i^{\pm}$ cobounds a {\em bridge disk} $\Delta_i^{\pm}$ with an arc in $S$.
We can take the bridge disks $\Delta_i^+$ $(i=1, \ldots, n)$ to be mutually disjoint,
and also for $\Delta_i^-$ $(i=1, \ldots, n)$.
By a {\em bridge surface}, we mean $S - K$.
The set of vertices of $\mathcal{D}(S - K)$ consists of 
compressing disks for $S - K$ in $V^+ - K$ and $V^- - K$.

Two bridge surfaces $S - K$ and $S' - K$ are equivalent if they are isotopic in $M- K$.
An $n$-bridge position of the unknot in $S^3$ is unique for every $n$ \cite{Otal},
so for $n \ge 2$ it is {\em perturbed}, i.e.
there exists a pair of bridge disks $\Delta_i^+$ and $\Delta_j^-$ such that $|\Delta_i^+ \cap \Delta_j^-| =1$.
The uniqueness holds also for $2$-bridge knots \cite{Scharlemann-Tomova} and torus knots \cite{Ozawa}.
However, there are $3$-bridge knots that admit infinitely many $3$-bridge spheres \cite{Jang}.

\section{Proof of Theorem \ref{thm}}\label{sec3}

Let $S^3$ be decomposed into two $3$-balls $B^+$ and $B^-$ with common boundary $S$.
Let $K$ be an unknot in $S^3$ which is in $(n+1)$-bridge position with respect to $S$.
Then $K \cap B^{\pm}$ is a collection of $n+1$ bridges $a_i^{\pm}$ $(i=1, \ldots, n+1)$ in $B^{\pm}$.
We assume that the bridges are arranged so that $a_1^{\pm}$ is adjacent to $a_1^{\mp}$ and $a_2^{\mp}$, and
$a_i^{\pm}$ is adjacent to $a_{i-1}^{\mp}$ and $a_{i+1}^{\mp}$ for $2 \le i \le n$, and
$a_{n+1}^{\pm}$ is adjacent to $a_n^{\mp}$ and $a_{n+1}^{\mp}$.
Let $\{ \Delta_i^{\pm} \}$ be a collection of disjoint bridge disks $\Delta_i^{\pm}$ for $a_i^{\pm}$
with $\Delta_i^{\pm} \cap S = b_i^{\pm}$.
%, and let $p_{i,1}^{\pm}$ and $p_{i,2}^{\pm}$ be the two points $\partial a_i^{\pm}(=\partial b_i^{\pm})$.
We assume that $\textrm{int}\, b_i^+ \cap \textrm{int}\, b_j^- = \emptyset$ for any $i$ and $j$.
See Figure $1$ for an example.

\begin{figure}[ht!]
\begin{center}
\includegraphics[width=12cm]{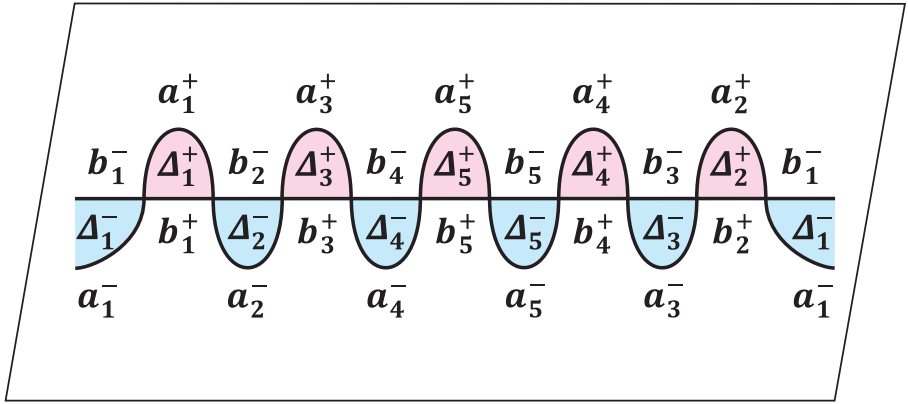}
\caption{Bridges and bridge disks.}\label{fig1}
\end{center}
\end{figure}

Let $P$ be the $(2n+2)$-punctured sphere $S - K$.
We define compressing disks $D_i^{\pm}$ $(i=1, \ldots, n)$ for $P$ in $B^{\pm} - K$ as follows.
Let $D_1^+$ be a disk in $B^+ - K$ such that $\partial D_1^+ = \partial N(b_1^+)$, where
$N(b_1^+)$ is a neighborhood of $b_1^+$ taken in $S$.
Similarly, other disks are defined so as to satisfy the following.
\begin{itemize}
\item $\partial D_1^- = \partial N(b_1^-)$
\item $\partial D_2^+ = \partial N(b_1^+ \cup b_1^- \cup b_2^+)$
\item $\partial D_2^- = \partial N(b_1^- \cup b_1^+ \cup b_2^-)$
\item[] $\vdots$
\item $\partial D_i^+ = \partial N(b_1^+ \cup b_1^- \cup \cdots \cup b_{i-1}^+ \cup b_{i-1}^- \cup b_i^+)$
\item $\partial D_i^- = \partial N(b_1^- \cup b_1^+ \cup \cdots \cup b_{i-1}^- \cup b_{i-1}^+ \cup b_i^-)$
\item[] $\vdots$
\item $\partial D_n^+ = \partial N(b_1^+ \cup b_1^- \cup \cdots \cup b_{n-1}^+ \cup b_{n-1}^- \cup b_n^+)$
\item $\partial D_n^- = \partial N(b_1^- \cup b_1^+ \cup \cdots \cup b_{n-1}^- \cup b_{n-1}^+ \cup b_n^-)$
\end{itemize}
In Figure $2$, $\partial D_i^{\pm}$'s in $P$ are depicted.

\begin{figure}[ht!]
\begin{center}
\includegraphics[width=13cm]{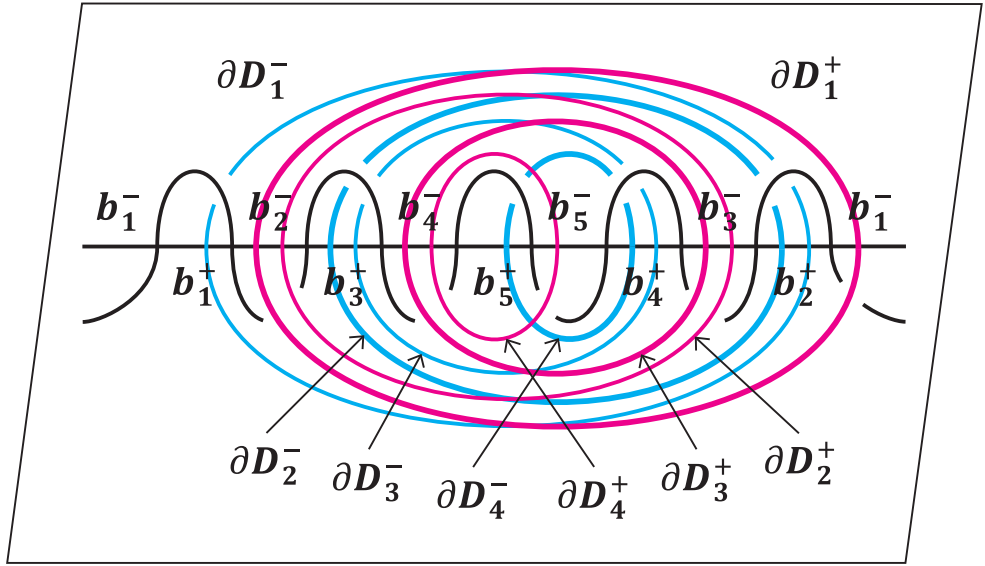}
\caption{$\partial D_i^{\pm}$ $(i=1, \ldots, n)$ in $P$.}\label{fig2}
\end{center}
\end{figure}

Now we define subsets $C_i^{\pm}$ $(i=1, \ldots, n)$ of the set of vertices of $\mathcal{D}(P)$ as follows.
For odd $i$, let
\begin{itemize}
\item $C_i^+ = \{ D_i^+ \}$
\item $C_i^- = \{$essential disks in $B^- - K$ that intersect $D_i^+$ and
are disjoint from $D_1^+, D_3^+, \ldots, D_{i-2}^+ \}$.
\end{itemize}
For even $i$, let
\begin{itemize}
\item $C_i^+ = \{$essential disks in $B^+ - K$ that intersect $D_i^-$ and
are disjoint from $D_2^-, D_4^-, \ldots, D_{i-2}^- \}$
\item $C_i^- = \{ D_i^- \}$.
\end{itemize}

Note that for all $i$, $D_i^{\pm}$ belongs to $C_i^{\pm}$.

\begin{lemma}
The collection $\{ C_i^{\pm} \}$ $(i=1, \ldots, n)$ is a partition of the set of essential disks in $B^{\pm} - K$.
\end{lemma}

\begin{proof}
First we show that $\{ C_i^+ \}$ $(i=1, \ldots, n)$ is a partition of the set of essential disks in $B^+ - K$.
We show that any essential disk in $B^+ - K$ belongs to one and only one $C_i^+$.

An essential disk in $B^+ - K$ that intersects $D_2^-$ belongs to $C_2^+$ by definition.
Let $E_2 = N(b_1^- \cup b_1^+ \cup b_2^-)$ be the disk in $S$ such that $\partial E_2 = \partial D_2^-$.

\begin{claim}\label{claim1}
If an essential disk $D$ in $B^+ - K$ is disjoint from $D_2^-$ and $\partial D$ is in $E_2$,
then $D$ is isotopic to $D_1^+ \in C_1^+$.
\end{claim}

\begin{figure}[ht!]
\begin{center}
\includegraphics[width=12cm]{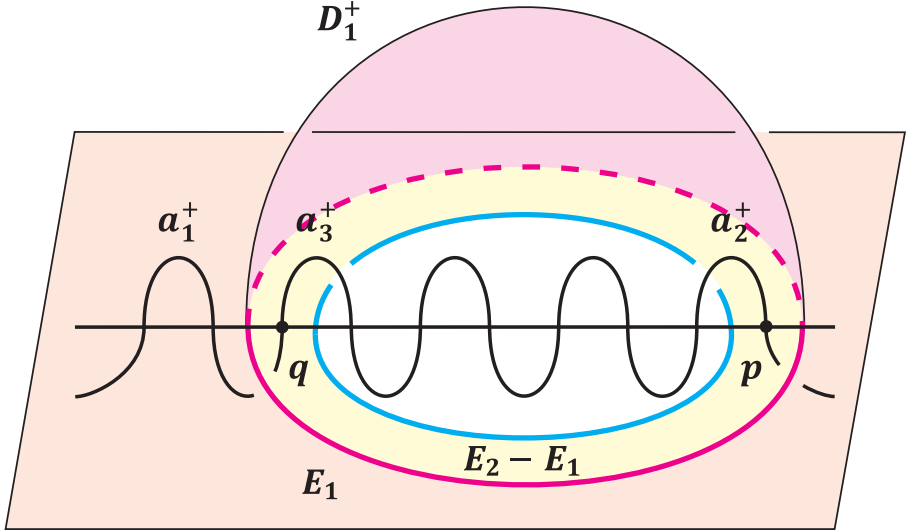}
\caption{$D_1^+$ in $C_1^+$.}\label{fig3}
\end{center}
\end{figure}

\begin{proof}[Proof of Claim \ref{claim1}]
We assume that $D$ intersects $D_1^+$ transversely and minimally,
so $D \cap D_1^+$ consists of arc components.
Let $E_1 = N(b_1^+)$ be the disk in $S$ such that $\partial E_1 = \partial D_1^+$.
See Figure $3$.
Suppose that $D \cap D_1^+ \ne \emptyset$.
Consider an outermost disk $\Delta$ of $D$ cut off by an outermost arc of $D \cap D_1^+$.
We can see that by minimality of $|D \cap D_1^+|$,
$\Delta$ cannot lie in the $3$-ball $B$ bounded by $D_1^+ \cup E_1$ containing $a_1^+$.
So $\Delta$ lies outside of $B$.
Let $\overline{D}$ be one of the disks obtained from $D_1^+$ by surgery along $\Delta$
such that $\partial\overline{D}$ bounds a disk $\overline{E}$ in $E_2 - E_1$.
Let $p$ be the point $a_2^+ \cap (E_2 - E_1)$ and $q$ be the point $a_3^+ \cap (E_2 - E_1)$.

Suppose $\overline{E}$ contains $p$.
Then the sphere $\overline{D} \cup \overline{E}$ intersects $a_2^+ \cup b_2^+$ in a single point
after a slight isotopy of $\mathrm{int}\, b_2^+$ into $B^-$, a contradiction.
So $\overline{E}$ does not contain $p$, and by similar reason $\overline{E}$ does not contain $q$.
Then $\overline{E}$ is an inessential disk in $E_2 - E_1 - K$, so we can reduce $|D \cap D_1^+|$, a contradiction.

Hence $D \cap D_1^+ = \emptyset$.
Let $E$ be the disk in $E_2$ such that $\partial E = \partial D$.
If $\partial E$ is in $E_1$, then $D$ is isotopic to $D_1^+$.
Suppose $\partial E$ is in $E_2 - E_1$.
Then $E$ contains neither $p$ nor $q$,
since otherwise $D \cup E$ intersects $a_2^+ \cup b_2^+$ or $a_3^+ \cup b_3^+$ in a single point as above.
So we get the conclusion that $D$ is isotopic to $D_1^+$.
\end{proof}

Therefore if an essential disk in $B^+ - K$ is disjoint from $D_2^-$ and its boundary is in $S - E_2$,
then it belongs to one of $C_3^+, \ldots, C_n^+$.

An essential disk in $B^+ - K$ that is disjoint from $D_2^-$ and intersects $D_4^-$
belongs to $C_4^+$ by definition.
Let $E_4 = N(b_1^- \cup b_1^+ \cup \cdots \cup b_3^- \cup b_3^+ \cup b_4^-)$ be the disk in $S$
such that $\partial E_4 = \partial D_4^-$.
Let $D$ be an essential disk in $B^+ - K$ that is disjoint from $D_2^-, D_4^-$ and $\partial D \subset S - E_2$.

\begin{claim}\label{claim2}
If $\partial D$ is in $E_4$ (hence in $E_4 - E_2$), then $D$ is isotopic to $D_3^+ \in C_3^+$.
\end{claim}

\begin{figure}[ht!]
\begin{center}
\includegraphics[width=12cm]{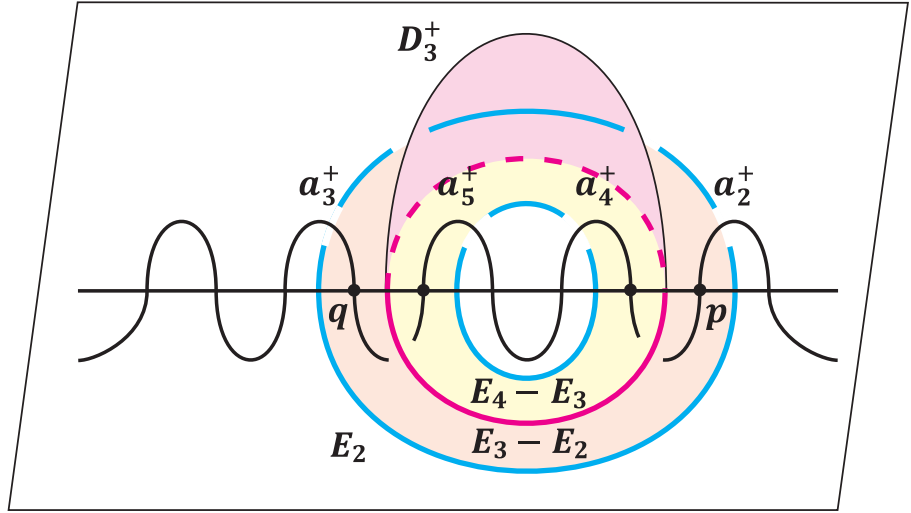}
\caption{$D_3^+$ in $C_3^+$.}\label{fig4}
\end{center}
\end{figure}

\begin{proof}[Proof of Claim \ref{claim2}]
We assume that $|D \cap D_3^+|$ is minimal up to isotopy, so $D \cap D_3^+$ consists of arc components.
Let $E_3 = N(b_1^+ \cup b_1^- \cup b_2^+ \cup b_2^- \cup b_3^+)$ be the disk in $S$ such that
$\partial E_3 = \partial D_3^+$.
See Figure $4$.
Suppose that $D \cap D_3^+ \ne \emptyset$.
Consider an outermost disk $\Delta$ of $D$ cut off by an outermost arc of $D \cap D_3^+$.
Without loss of generality, we assume that $\partial\Delta \cap S$ lies in $E_3 - E_2$.
Let $\overline{D}$ be one of the disks obtained from $D_3^+$ by surgery along $\Delta$
such that $\partial\overline{D}$ bounds a disk $\overline{E}$ in $E_3 - E_2$.
Let $p$ be the point $a_2^+ \cap (E_3 - E_2)$ and $q$ be the point $a_3^+ \cap (E_3 - E_2)$.

Suppose $\overline{E}$ contains $p$.
Then the sphere $\overline{D} \cup \overline{E}$ intersects $a_2^+ \cup b_2^+$ in a single point
after a slight isotopy, a contradiction.
So $\overline{E}$ does not contain $p$, and similarly $\overline{E}$ does not contain $q$.
Then $\overline{E}$ is an inessential disk in $E_3 - E_2 - K$, so we can reduce $|D \cap D_3^+|$, a contradiction.
Hence $D \cap D_3^+ = \emptyset$.
Then as in Claim \ref{claim1} we can see that $D$ is isotopic to $D_3^+$.
\end{proof}

Therefore if an essential disk in $B^+ - K$ is disjoint from $D_2^-$, $D_4^-$ and its boundary is in $S - E_4$,
then it belongs to one of $C_5^+, \ldots, C_n^+$.

In general, let $E_{2i} = N(b_1^- \cup b_1^+ \cup \cdots \cup b_{2i-1}^- \cup b_{2i-1}^+ \cup b_{2i}^-)$
be the disk in $S$ such that $\partial E_{2i} = \partial D_{2i}^-$.
Let $D$ be an essential disk in $B^+ - K$ that is disjoint from $D_2^-, D_4^-, \ldots, D_{2i-2}^-$ and
$\partial D \subset S - E_{2i-2}$.

\begin{itemize}
\item If $\partial D \subset E_{2i} - E_{2i-2}$, then $D$ is isotopic to $D_{2i-1}^+ \in C_{2i-1}^+$.
\item If $D$ intersects $D_{2i}^-$, then $D$ belongs to $C_{2i}^+$ by definition.
\item If $\partial D \subset S - E_{2i}$, then $D$ belongs to one of $C_{2i+1}^+, \ldots, C_n^+$.
\end{itemize}

An inductive argument in this way leads to that any essential disk in $B^+ - K$ belongs to one and only one $C_i^+$.
A similar argument shows that $\{ C_i^- \}$ $(i=1, \ldots, n)$
is a partition of the set of essential disks in $B^- - K$.
\end{proof}

The collection of disks $\{ D_1^+, D_1^-, \ldots, D_n^+, D_n^- \}$ spans an $(n-1)$-sphere $S^{n-1}$ in $\mathcal{D}(P)$.
There is no edge in $\mathcal{D}(P)$ connecting $C_i^+$ and $C_i^-$ by definition.
There exists an edge in $\mathcal{D}(P)$ connecting $C_i^{\pm}$ and $C_j^{\pm}$ for $i \ne j$,
e.g. en edge between $D_i^{\pm}$ and $D_j^{\pm}$,
and there exists an edge in $\mathcal{D}(P)$ connecting $C_i^+$ and $C_j^-$ for $i \ne j$,
e.g. an edge between $D_i^+$ and $D_j^-$.
Hence if we define a map $\overline{r}$ from the set of vertices of $\mathcal{D}(P)$ to the set of vertices of $S^{n-1}$ by
$$\overline{r}(v) = D_i^{\pm} \quad \textrm{if}\,\, v \in C_i^{\pm},$$
then $\overline{r}$ extends to a continuous map from the $1$-skeleton of $\mathcal{D}(P)$ to the $1$-skeleton of $S^{n-1}$.
Since higher dimensional simplices of $\mathcal{D}(P)$ are determined by $1$-simplices,
$\overline{r}$ can be extended to a retraction $r : \mathcal{D}(P) \to S^{n-1}$.
Hence $\pi_{n-1}(\mathcal{D}(P)) \ne 1$, and
the topological index of $P$ is at most $n$.

%\noindent {\bf Acknowledgements.}
%The author was supported by the National Research Foundation of Korea Grant
%funded by the Korean Government (NRF-2013R1A1A2059197).

\end{document}